\documentclass[11,review]{amsart}

\usepackage{amsfonts,amssymb,amsmath}

\usepackage[cp1251]{inputenc}
\usepackage[T2B]{fontenc}
\usepackage[english]{babel}
\usepackage[mathscr]{eucal}
\usepackage{calrsfs}

\newtheorem{theorem}{Theorem}[section]
\newtheorem{lemma}[theorem]{Lemma}

\newtheorem{corollary}[theorem]{Corollary}

\theoremstyle{definition}

\theoremstyle{remark}
\newtheorem{remark}[theorem]{Remark}

\numberwithin{equation}{section}

\DeclareMathOperator*{\esup}{ess\,sup}

\begin{document}\large
\title[On upper estimates for approximation numbers]{On upper estimates for approximation numbers of a Laplace type transformation}

\author{Elena P. Ushakova}

\curraddr{Department of Mathematics, University of York, York, YO10 5DD, UK.}\email{elena.ushakova@york.ac.uk}

\address{Computing Centre of the Far-Eastern Branch of the Russian Academy of Sciences, Khabarovsk, 680000, RUSSIA.}
\email{elenau@inbox.ru}

\begin{abstract}
We deal with a real valued integral operator $\mathcal{L}$ of Laplace transformation type acting between Lebesgue spaces on the semi-axis. Sufficient conditions for belonging $\mathcal{L}$ to Schatten type classes are obtained. Some upper asymptotic estimates for the approximation numbers of $\mathcal{L}$ are also given.
\end{abstract}



\maketitle

\section{Introduction and preliminaries}

Given an operator $T$ between arbitrary (quasi-)Banach spaces $X$ and $Y$ $$a_n(T)=\inf\{\|T-K\|_{X\to Y}\colon K: X\to Y,\ \textrm{rank}\,K<n\}$$ is called the $n$--th approximation number of $T.$ In a case of Hilbert spaces these numbers coincide with singular values of $T.$ Compact operators $T:X\to Y$ satisfying \begin{equation}\label{SN} \|T\|_{\mathbf{S}_\alpha}=\biggl(\sum_{n=1}^\infty a_n^\alpha(T)\biggr)^{1/\alpha}<\infty,\ \ \ \ \ 0<\alpha<\infty,\end{equation} constitute Schatten-von Neumann classes $\mathbf{S}_\alpha.$ Symbol $\mathbf{S}_{\alpha,\textrm{weak}}$ stands for weak Schatten-von Neumann classes consisting of all operators $T$ such that $$\|T\|_{\mathbf{S}_{\alpha,\textrm{weak}}}=\sup_{t>0} \bigl(t(\sharp\{n\ge 1\colon a_n(T)>t\}\bigr)^{1/\alpha}\bigr)<\infty,\ \ \ \ \ 0<\alpha<\infty.$$ Schatten-Lorentz classes $\mathbf{S}_{\alpha,\beta}$ and $\mathbf{S}_{\alpha,\infty}$ are defined by \begin{equation}\label{SL} \mathbf{S}_{\alpha,\beta}=\biggl\{ T\colon \sum_{n=1}^\infty a_n^\beta(T)\,n^{\beta/\alpha-1}<\infty\biggr\}\ \ \ \ \ 0<\alpha,\beta<\infty,\end{equation} and $$\mathbf{S}_{\alpha,\infty}=\biggl\{ T\colon a_n(T)\le \textrm{const}\,n^{-1/\alpha}\biggr\}\ \ \ \ \ 0<\alpha<\infty.$$

Let $L^r(I)$ denote a space of all functions $f$ measurable in the Lebesgue sense on $I\subseteq[0,+\infty)=:\mathbb{R}^+$ with $\|f\|_{r,I}:=\bigl(\int_I |f|^r\bigr)^{1/r}<\infty.$
We study sequences of approximation numbers of an integral operator
\begin{equation}\label{L}
\mathcal{L}f(x):=\int_{\mathbb{R}^+}\mathrm{e}^{-xy^\lambda}f(y)v(y)\mathrm{d}y,\hspace{5mm}
x\in\mathbb{R}^+,\hspace{1cm}\mathcal{L}\colon L^p\to L^q,\end{equation}
when $1\le p<\infty,$ $0<q<\infty,$ $\lambda>0$ and a non-negative weight function $v$ is locally integrable on $\mathbb{R}^+.$

Boundedness and compactness properties of the Laplace transform $\mathcal{L}$ in the Lebesgue spaces were completely studied in \cite{SUsmz} and \cite{U2011}.

\begin{theorem}\label{BC}\cite[Th.1]{SUsmz},\cite[Th.3.1]{U2011} \\
{\rm\textbf{(i)}} Let $1<p\le q<\infty,$ $p':=p/(p-1)$ and $q':=q/(q-1).$ Denote $\kappa_1:=q^{-2/q}[\min\{2,2^{q-1}\}]^{1/q}$ and ${\bar \kappa}_1:=\begin{cases} 2^{1/q}(q')^{1/p'}(q-1)^{-1/q}, & 1<q\le 2\\ 2^{1/q'}(q')^{1/p'}, & q>2\end{cases}.$\\ The operator $\mathcal{L}$ is $L^p-L^q$-bounded if and only if
$$A_\mathcal{L}:=\sup_{t\in\mathbb{R}^+}A_\mathcal{L}(t)= \sup_{t\in\mathbb{R}^+}t^{-\lambda/q}\biggl(\int_{0}^{t}v^{p'}(y)dy\biggr)^{1/p'}<\infty,$$ where
$\kappa_1 A_\mathcal{L}\le \|\mathcal{L}\|_{L^p\to L^q}\le{\bar \kappa}_1 A_\mathcal{L}.$ $\mathcal{L}:L^p\to L^q$ is compact if and only if $A_{\mathcal{L}}<\infty$ and \begin{equation}\label{comp}{\rm(i)}\,\lim_{t\to 0}A_\mathcal{L}(t)=0,\ \ \ \ \  {\rm(ii)}\lim_{t\to \infty}A_\mathcal{L}(t)=0.\end{equation}
{\rm\textbf{(ii)}} Let $1\le q<p<\infty$ and $r:=p\,q/(p-q).$ If $q=1$ then $\mathcal{L}$ is bounded if and only if $\|\mathcal{L}\|_{L^p\to L^1}=B_p:=\bigl(\int_{\mathbb{R}^+}y^{-\lambda p'}v^{p'}(y)dy\bigr)^{1/p'}<\infty.$ If $q>1$ then $\mathcal{L}$ is bounded if and only if $$B_\mathcal{L}:=\biggl(\int_{\mathbb{R}^+} t^{-\lambda r/q}\biggl[\int_0^tv^{p'}(y)dy\biggr]^{r/q'}v^{p'}(t)dt\biggr)^{1/r}<\infty,$$ where
$\kappa_2 B_\mathcal{L}\le \|\mathcal{L}\|_{L^p\to L^q}\le {\bar \kappa}_2 B_\mathcal{L}$ with $\kappa_2:=[\min\{2,2^{q-1}\}/q]^{1/q}(p'q/r)^{1/q'},$ ${\bar \kappa}_2:=2^{1/q}(p')^{1/q'}(q-1)^{-1/q}$ if
$ 1<q\le 2$ and ${\bar \kappa}_2:=2^{1/q'}(p')^{1/q'}$ for $q>2.$ Moreover, since $1\le q<p<\infty$ then the boundedness of the operator $\mathcal{L}:L^p\to L^q$ is equivalent to the compactness of $\mathcal{L}$ from $L^p$ to $L^q.$\\
{\rm\textbf{(iii)}} If $0<q<1<p<\infty$ then $\mathcal{L}$ is bounded if $B_\mathcal{L}<\infty.$ If $\mathcal{L}$ is bounded then $\|B_{q}\|_{p'}:=\bigl(\int_{\mathbb{R}^+}y^{-\lambda p'/q}v^{p'}(y)dy\bigr)^{1/p'}<\infty.$ Besides, \begin{equation*}q^{-1/q}\|B_{q}\|_{p'}\le \|\mathcal{L}\|_{L^p\to L^q}\le p^{1/p}(p')^{1/q'}q^{-2/q}r^{1/r} B_\mathcal{L},\end{equation*} and the $L^p-L^q$- boundedness of $\mathcal{L}$ is equivalent to the compactness of $\mathcal{L}$ from $L^p$ to $L^q$ provided $q<p.$\\
{\rm\textbf{(iv)}} Let $0<q<1=p.$ The transformation $\mathcal{L}$ is bounded if $$B_{q'}:=\biggl(\int_{\mathbb{R}^+}t^{-\lambda/(1-q)-1}[\esup_{0<x<t}v(x)]^{q/(1-q)}dt\biggr)^{(1-q)/q}<\infty.$$ If $\mathcal{L}$ is bounded then $\displaystyle\esup_{t\in\mathbb{R}^+}B_q(t):=
\esup_{t\in\mathbb{R}^+}t^{-\lambda/q}\esup_{0<x<t}v(x)<\infty.$ We have also
\begin{eqnarray*}q^{-1/q}\esup_{t\in\mathbb{R}^+}B_q(t)\le \|\mathcal{L}\|_{L^1\to L^q}
\le  \lambda^{(1-q)/q}q^{-2/q} (1-q)^{-(1-q)/q} B_{q'}.\end{eqnarray*} Moreover, in view of $q<p=1$ the $L^1-L^q$- boundedness of $\mathcal{L}$ is equivalent to the $L^1-L^q$-compactness of $\mathcal{L}.$\\
{\rm\textbf{(v)}} If $1=p\le q<\infty$ then $\mathcal{L}$ bounded iff $\displaystyle\esup_{t\in\mathbb{R}^+}B_q(t)<\infty.$ The operator is compact from $L^1$ to $L^q$ for $1\le q<\infty$ iff $\displaystyle\esup_{t\in\mathbb{R}^+}B_q(t)<\infty$ and $\lim_{t\to 0}B_q(t)=\lim_{t\to\infty}B_q(t)=0.$\\
{\rm\textbf{(vi)}} Let $p=\infty.$ If $1\le q<\infty$ then $\mathcal{L}$ is bounded if and only if $$C_q:=\biggl(\int_{\mathbb{R}^+} t^{-\lambda}\biggl[\int_0^t v(y)dy\biggr]^{q-1}v(t)dt\biggr)^{1/q}<\infty.$$ For $q<1$ the operator $\mathcal{L}$ is bounded if $C_q<\infty.$ If $\mathcal{L}$ is bounded then $\int_{\mathbb{R}^+}t^{-\lambda/q}v(t)dt<\infty.$ Besides, for all $0<q<\infty$ the $L^\infty-L^q$-boundedness of the Laplace transformation $\mathcal{L}$ is equivalent to the compactness of $\mathcal{L}.$\\
{\rm\textbf{(vii)}} Let $q=\infty.$ If $1<p\le\infty$ then $\mathcal{L}$ is bounded iff $\|v\|_{p'}<\infty$ and compact iff $\|v\|_{p'}<\infty.$  If $p=1$ then $\mathcal{L}$ is never compact, but bounded iff $\esup_{t\in\mathbb{R}^+}B_1(t)<\infty.$
\end{theorem}

Schatten-von Neumann norms of $\mathcal{L}:L^2\to L^2$ were studied in \cite{SMZ10} (see also \cite{DAN10}).
\begin{theorem}\label{cite}\cite[Th.4]{SMZ10}, \cite[Th.3]{DAN10} Suppose $\mathcal{L}$ is compact from $L^2$ to $L^2.$\\
{\rm\textbf{(i)}} If $\mathcal{L}\in\mathbf{S}_\alpha$ with $2\le \alpha<\infty$ then $$X_\alpha:=\biggl(\int_{\mathbb{R}^+}x^{-(\lambda\alpha/2+1)}\biggl(\int_0^x v^2(y)\mathrm{d}y\biggr)^{\alpha/2}\mathrm{d}x\biggr)^{1/\alpha}<\infty.$$
{\rm\textbf{(ii)}} If $\mathcal{L}\in\mathbf{S}_\alpha$ for $0<\alpha\le 2$ then $X_2<\infty.$\\
{\rm\textbf{(iii)}} If $X_\alpha<\infty$ for $0<\alpha<\infty$ then $\mathcal{L}\in\mathbf{S}_\alpha$ for all $0<\alpha<\infty.$\end{theorem} \begin{remark} The assertion {\rm\textbf{(i)}} of the theorem can be improved to $1\le \alpha<\infty.$ Indeed, by \cite[Lemma 2.11.12]{Pie} we have for all orthonormal sequences $\{f_k\}$ and $\{g_k\}$ \begin{equation}\label{hil}\sum_{k=1}^n|(\mathcal{L}f_k,g_k)|^\alpha \le \sum_{k=1}^n a_k^\alpha(\mathcal{L}),\ \ \ \ \ \ \ \ n\le\infty,\end{equation} where $1\le\alpha<\infty,$ $\mathcal{L}$ is compact on $L^2$ and $\langle\cdot,\cdot\rangle$ is the inner product. Applying the functions $$f_k(y)=\chi_{(2^{k-1},2^{k})}(y)v(y)\biggl(\int_{2^{k-1}}^{2^k}v^2(z)\mathrm{d}z\biggr)^{-1/2},$$ $$g_k(x)=\frac{2^{k\lambda/2+\lambda}}{2^\lambda{-1}}\chi_{(2^{-(k+1)\lambda},2^{-k\lambda})}(x)$$ into the left-hand side of \eqref{hil} we obtain \begin{eqnarray*}\sum_{k\in\mathbb{N}}a_k^\alpha\ge \frac{1}{2}\sum_{k\in\mathbb{Z}}\tau_k^\alpha, \end{eqnarray*}
where $\tau_k=2^{-\lambda k/2}\biggl(\int_{2^{k-1}}^{2^k}v^2(z)\mathrm{d}z\biggr)^{1/2}$ and $\bigl(\sum_{k\in\mathbb{Z}}\tau_k^\alpha\bigr)^{1/\alpha}\approx X_\alpha,$ $0<\alpha<\infty.$
Thus, the parts {\rm\textbf{(i)}} and {\rm\textbf{(ii)}} of the theorem \ref{cite} have been changed as follows:\\
{\rm\textbf{(i')}} If $\mathcal{L}\in\mathbf{S}_\alpha,$ where $1\le \alpha<\infty,$ then $X_\alpha<\infty$ for $1\le\alpha<\infty.$\\
{\rm\textbf{(ii')}} If $\mathcal{L}\in\mathbf{S}_\alpha$ for $0<\alpha\le 1,$ then $X_1<\infty.$
\end{remark}

This work is devoted to approximation numbers of the operator $\mathcal{L}$ in (quasi-)Banach case of Lebesgue spaces. We obtain sufficient conditions for the Laplace transformation to belong to weighted Schatten classes $$\mathbf{S}_{s,u}=\biggl\{T\colon \biggl(\sum_{n=1}^\infty \bigl[a_n(\mathcal{L})u_n\bigr]^s\biggr)^{1/s}<\infty\biggr\},$$ when the operator is acting from $L^p$ to $L^q$ with $1\le p<\infty$ and $0<q<\infty.$  The results imply upper estimates either for the corresponding Schatten-Lorentz norms \eqref{SL} or for the Schat\-ten-von Neumann classes \eqref{SN}.
Some upper asymptotic estimates for the sequences $\{a_n(\mathcal{L})\}$ are also given in the article.

Approximation numbers of integral operators have being intensively studied since late seventies of the last century. In particular, in some recent papers (see e.g. \cite{BirSol} -- 
  \cite{ES1994}, \cite{LL1999} -- 
  \cite{LS2007}, \cite{S2000}) authors deal with Volterra type integral operators. The results in \cite{BirSol}, \cite{ES1994}, \cite{LS2007} and \cite{S2000} are even on the operators with non-factorised kernels. The purpose of our work is study approximation numbers of an integral operator of such a type.

For achieving our results we mainly adapt remarkable methods which were originated in \cite{EEH1988}, \cite{EEH1997}, \cite{EHL1998} and continued in \cite{LS2000} for Hardy integral operator $H_{\mathbb{R}^+,v,w}$ defined by \eqref{hardyy}. The transformation $\mathcal{L}$ and the Hardy operator are related to each other by Lemma 1 from \cite{SUsmz}, which is true for subclasses of non-negative functions from Lebesgue spaces only. By using this fact we extract some upper estimates for approximation numbers of the operator $\mathcal{L}.$

For our purposes Lemma 1 \cite{SUsmz} has to be modified as follows.
\begin{lemma}\label{dec}\cite{SUsmz}
Let $0\le a<b\le \infty$ and an operator $\mathcal{L}_{(a,b)}$ be given by \begin{equation}\label{L_a}\mathcal{L}_{(a,b)}f(x):=\int_{a}^bf(y)v(y) \bigl[\mathrm{e}^{-xy^\lambda}-\mathrm{e}^{-xb^\lambda}\bigr]\mathrm{d}y,\ \ \ x>0.\end{equation} If $f\ge 0$ and $1<q<\infty$ then \begin{eqnarray}
\frac{\alpha_0}{q}\int_a^b\biggl(\int_a^z f(y)v(y)\mathrm{d}y\biggr)^q\mathrm{d}\biggl[-\biggr(\int_{\mathbb{R}^+} \bigl[\mathrm{e}^{-xz^\lambda}-\mathrm{e}^{-xb^\lambda}\bigr]^q\mathrm{d}x\biggr)\biggr]\nonumber\\\le \|\mathcal{L}_{(a,b)}f\|_q^q
\le\frac{\beta_0}{q}\int_a^b\biggl(\int_a^z f(y)v(y)\mathrm{d}y\biggr)^q\mathrm{d}\bigl[-(z^{-\lambda}-b^{-\lambda})\bigr],
\end{eqnarray} where $\alpha_0=\max\{2,2^{q-1}\}$ and $\beta_0:=\begin{cases}2/(q-1), & 1<q\le 2\\ 2^{q-1}, & q>2\end{cases}.$\end{lemma}

The article is organised as follows. In Section \ref{an} we obtain preliminary estimates for $a$-numbers of $\mathcal{L}.$ Section \ref{sn} is devoted to the norms of Schatten type (see Theorems \ref{main} and \ref{main'}). Section \ref{ab} is on asymptotic estimates (see Theorem \ref{kr}).

Throughout  the article products of the form $0\cdot\infty$ are supposed to be equal to 0. We write $A\ll B$ or $A\gg B$ when $A\le c_1 B$ or $A\ge c_2 B$ with constants $c_i,$ $i=1,2,$ depending on $\lambda,p,q,s$ only. $A\approx B$ means $A\ll B\ll A.$ Symbols $\mathbb{Z}$ and $\mathbb{N}$ denote integers $\{k\}$ and naturals $\{n\}$ respectively. $\chi_E$ stands for a characteristic function of a subset $E\subset\mathbb{R}^+.$ We also use $=:$ and $:=$ for marking new quantities.

\section{Approximation numbers}\label{an}
Let $(a,b)=:I\subseteq\mathbb{R}^+$ and $\mathcal{L}_b$ be an operator given by $$\mathcal{L}_bf(x):=\mathrm{e}^{-xb^\lambda} \int_{\mathbb{R}^+}f(y)v(y)\mathrm{d}y.$$ Put $$ K(I):=\bigl\{\sup \|\mathcal{L}(f\chi_I)-\mathcal{L}_b(f\chi_I)\|_{q}/\|f\|_{p,I}\colon f\in L^{p}(I)\bigr\}$$ and notice that in view of the denotation \eqref{L_a}
$$ K(I)=\bigl\{\sup\|\mathcal{L}_If\|_{q}/\|f\|_{p,I}\colon f\in L^{p}(I)\bigr\}.$$

We start from \subsection{Case $1\le p,q<\infty$}
Following the classical scheme for the Hardy integral operator (see \cite{EEH1988}, \cite{LS2000}) we define quantities
\begin{equation*}
A_{0,I}:=A_0(\delta)_I=\sup_{t\in I}
\biggl(\int_{a}^t v^{p'}(y)\mathrm{d}y\biggr)^{1/p'}
\biggr(\int_{\mathbb{R}^+} \bigl[\mathrm{e}^{-xt^\lambda}-\mathrm{e}^{-xb^\lambda}\bigr]^\delta\mathrm{d}x\biggr)^{1/q},\end{equation*}
\begin{equation*}B_{0,I}:=
B_0(\delta)_I=\biggl(\int_I
\biggl(\int_{a}^t v^{p'}(y)\mathrm{d}y\biggr)^{r/p'}
\mathrm{d}\biggl[-\biggr(\int_{\mathbb{R}^+} \bigl[\mathrm{e}^{-xt^\lambda}-\mathrm{e}^{-xb^\lambda}\bigr]^\delta\mathrm{d}x\biggr)^{r/q} \biggr]\biggr)^{1/r},\end{equation*}
\begin{equation*}
A_{1,I}=\esup_{t\in I}\ \esup_{a<y<t}v(y)\bigl[t^{-\lambda}-b^{-\lambda}\bigr]^{1/q},\ \ \
B_{1,I}=\biggl(\int_I \bigl[t^{-\lambda}-b^{-\lambda}\bigr]^{p'}v^{p'}(t) \mathrm{d}t\biggr)^{1/p'}\end{equation*} with $r=pq/(p-q),$ which sandwich the norm $K(I)$ as follows.

\begin{lemma}\label{lemma1} Let $1\le p,q<\infty.$ We have
\begin{eqnarray}\label{z}\gamma_0A_0(q)_I\le K(I)\le {\bar \gamma}_0 A_0(1)_I,\ \ \ \ \ 1<p\le q<\infty,\\
\gamma_1A_{1,I}\le K(I)\le {\bar \gamma}_1 A_{1,I},\ \ \ \ \ 1=p\le q<\infty,\\\label{zz}
\gamma_2 B_0(q)_I\le K(I)\le {\bar \gamma}_2B_0(1)_I,\ \ \ \ \ 1<q<p<\infty,\\\label{4}
  K(I)=B_{1,I},\ \ \ \ \ 1=q<p<\infty,\end{eqnarray} where $\gamma_0=\gamma_1=\alpha_0^{1/q}q^{-1/q},$ ${\bar \gamma}_0=\beta_0^{1/q}(q')^{1/p'},$ ${\bar \gamma}_1=\beta_0^{1/q}(q')^{-1/q},$
  $\gamma_2=\alpha_0^{1/q}(qp'/r)^{1/q'},$ ${\bar \gamma}_2=\beta_0^{1/q}(p')^{1/q'}.$  \end{lemma}

\begin{proof} Let $1\le p\le q<\infty,$ $A_{i,I}<\infty,$ $i=0,1,$ and $f\in L^{p}(I).$ By Lemma \ref{dec}
\begin{eqnarray}\label{AB}
\|\mathcal{L}_If\|_{q}^q\le \frac{\beta_0}{q}
\int_I\biggl(\int_a^z |f(y)|v(y)\mathrm{d}y\biggr)^{q}\mathrm{d} \bigl[-(z^{-\lambda}-b^{-\lambda})\bigr] \nonumber\\\le {\bar \gamma_i}^q A_{i,I}^q\|f\|_{p,I}^q.\end{eqnarray}  For the reverse estimate we assume $f\ge 0$ and pass from \begin{eqnarray*}\|\mathcal{L}_If\|_{q}\le K(I)\|f\|_{p,I}\end{eqnarray*} to the inequality
\begin{eqnarray}\label{Har}\frac{\alpha_0}{q}\int_{I}\biggl(\int_a^z f(y)v(y)\mathrm{d}y\biggr)^q\mathrm{d}\biggl[-\biggr(\int_{\mathbb{R}^+} \bigl[\mathrm{e}^{-xz^\lambda}-\mathrm{e}^{-xb^\lambda}\bigr]^q\mathrm{d}x\biggr)\biggr]\nonumber\\\le K(I)^q\|f\|_{p,I}^q\end{eqnarray} by using Lemma \ref{dec}. Thus, $K(I)\ge \gamma_0 A_0(q)_I.$ We also have $K(I)\ge \gamma_1 A_{1,I}$ when $p=1.$

If $1<q<p<\infty$ then similar to \eqref{AB}
\begin{eqnarray}\label{09}\|\mathcal{L}_If\|_{q}^q
\le {\bar \gamma}_2^q B_0(1)^{q}_I\|f\|_{p,I}^{q}.\end{eqnarray}
 Lemma \ref{dec} and arguments for the Hardy integral operator applied to the left-hand side of the inequality \eqref{Har} give us $K(I)\ge \gamma_2 B_0(q)_I,$ $q>1.$

For $1=q<p<\infty$ we obtain \eqref{4} by direct and reverse H\"{o}lder's inequalities.
\end{proof}

Assume that $\mathcal{L}$ is compact. In view of local integrability of $v$, on the strength of the conditions \eqref{comp} and by Lemma \ref{lemma1} the quantity $K(I)$ continuously depends on an interval $I\subseteq\mathbb{R}^+.$ Thus, any given $\varepsilon>0$ such that $0<\varepsilon<\|\mathcal{L}\|_{p\to q}$ we can find $N=N(\varepsilon)$ and, therefore, points $0=c_0<c_1<\ldots <c_{N+1}=\infty$ to form intervals $I_n=[c_n,c_{n+1}],$ $k=0,\ldots,N,$ so that the norm $K(I_n)$ is equal to $\varepsilon$ for all $k=0,\ldots,N-1$ and $K(I_N)\le\varepsilon.$

\begin{lemma}\label{lemma2} Let $1\le p,q<\infty$ and $0<\varepsilon\le\|\mathcal{L}\|_{p\to q}.$ Suppose there exists $N=N(\varepsilon)<\infty$ and points $0=c_0<c_1<\ldots <c_{N+1}=\infty$ such that $K(I_n)=\varepsilon$ for all $I_n=(c_n,c_{n+1}),$ $n=0,\ldots,N-1,$ and $K(I_N)\le\varepsilon.$ Then
\begin{equation}\label{up1}
a_{N+1}(\mathcal{L})\le \begin{cases}\varepsilon, & p=1,\\ \varepsilon(N+1)^{1/p'}, & p>1.\end{cases}\end{equation}
 \end{lemma}
\begin{proof} Let $f\in L^{p}$ be such that $\|f\|_{p}=1,$ and define $P\colon L^{p}\to L^{q}$ by \begin{equation}\label{P}
Pf(x)=\sum_{n=0}^{N}\mathcal{L}_{c_{n+1}}(f\chi_{I_n})(x).
\end{equation} Since $\mathcal{L}_{c_{N+1}}(f\chi_{I_n})\equiv 0$ then $\textrm{rank}\,P\le N.$ 
We have \begin{eqnarray*}
\|\mathcal{L}f-Pf\|_{q}=\|\sum_{n=0}^N[\mathcal{L}(f\chi_{I_n})-\mathcal{L}_{c_{n+1}}(f\chi_{I_n})]\|_{q}\le
\sum_{n=0}^N\|\mathcal{L}_{I_n}(f\chi_{I_n})\|_{q}\\
\le \sum_{n=0}^{N} K(I_n)\|f\|_{p,I_n}\le \varepsilon \sum_{n=0}^{N}\|f\|_{p,I_n}.
\end{eqnarray*} This immediately implies \eqref{up1} for $p=1.$ If $p>1$ then \eqref{up1} follows by H\"{o}lder's inequality.
 \end{proof}

\subsection{Case $0<q<1\le p<\infty$} In this part we operate with the constant $B_0(\delta)_I,$ when $\delta=q,$ and quantities $r=pq/(p-q),$
\begin{equation*}
B_{2,I}=\biggl(\int_I v^{p'}(y) \biggl(\int_{\mathbb{R}^+} \bigl[\mathrm{e}^{-xy^\lambda}-\mathrm{e}^{-xb^\lambda}\bigr]^q\mathrm{d}x\biggr)^{p'/q} \mathrm{d}y\biggr)^{1/p'},\end{equation*}
\begin{equation*}
B_{3,I}=\biggl(\int_I[\esup_{a<x<t}v(x)]^{q/(1-q)}\mathrm{d}\biggl[-\biggr(\int_{\mathbb{R}^+} \bigl[\mathrm{e}^{-xt^\lambda}-\mathrm{e}^{-xb^\lambda}\bigr]^q\mathrm{d}x\biggr)^{1/(1-q)} \biggr]\biggr)^{(1-q)/q},\end{equation*}
\begin{equation*}
B_{4,I}=\esup_{t\in I} \biggl[\esup_{a<x<t}v(x)\biggr(\int_{\mathbb{R}^+} \bigl[\mathrm{e}^{-xt^\lambda}-\mathrm{e}^{-xb^\lambda}\bigr]^q\mathrm{d}x\biggr)^{1/q}\biggr].\end{equation*}

The analog of Lemma \ref{lemma1} for $0<q<1\le p<\infty$ reads
\begin{lemma}\label{lemma03} Let $0<q<1\le p<\infty.$ We have
\begin{eqnarray}\label{5}B_{2,I}\le K(I)\le {\bar \gamma_3} B_0(q)_I, \ \ \ \ \ 0<q<1<p<\infty,\\\label{6}
B_{4,I}\le K(I)\le {\bar \gamma}_4B_{3,I}, \ \ \ \ \ 0<q<1=p,\end{eqnarray} where  ${\bar \gamma}_3=(r/q)^{1/r-1}p^{1/p}(p')^{1/p'}q^{-1/q}$ and ${\bar \gamma}_4=(1-q)^{-(1-q)/q}.$
\end{lemma}

\begin{proof} Since $q<1$ we have  by a monotonicity  argument
\begin{eqnarray*}
\|\mathcal{L}_If\|_{q}^q\le\frac{1}{q}
\int_a^c \biggl(\int_a^z |f(y)|v(y)\mathrm{d}y\biggr)^{q}\mathrm{d} \biggl[-\biggr(\int_{\mathbb{R}^+} \bigl[\mathrm{e}^{-xz^\lambda}-\mathrm{e}^{-xc^\lambda}\bigr]^q\mathrm{d}x\biggr) \biggr]\end{eqnarray*} (see \cite[p. 1131]{SUsmz} for details). Therefore, \begin{eqnarray*}
K(I)^q\le\begin{cases} {\bar \gamma}_3^q B_0^q(q)_I\|f\|_{p,I}^q, &p>1\\
{\bar \gamma}_4^q B_{3,I}^q\|f\|_{1,I}^q, &p=1\end{cases}.\end{eqnarray*}
For the reverse estimate we assume $f\ge 0$ and complete the proof of \eqref{5} and \eqref{6} by using Minkovskii's and H\"{o}lder's inequalities.\end{proof}

If $\mathcal{L}$ is compact and $0<\varepsilon<\|\mathcal{L}\|_{p\to q}$ then by Lemma \ref{lemma03} the norm $K(I),$ as before, continuously depends on an interval $I.$ Thus, for a sufficiently small $\varepsilon>0$ we can find points $0=c_0<c_1<\ldots <c_{N+1}=\infty,$ $N=N(\varepsilon),$ with $K(I_n)=\varepsilon$ for $n=0,\ldots,N-1$ and $K(I_N)\le\varepsilon.$

\begin{lemma}\label{lemma04}
Let $0<q<1\le p<\infty.$ Let $0<\varepsilon\le\|\mathcal{L}\|_{p\to q}$ and suppose there exists $N(\varepsilon)<\infty$ and points $0=c_0<c_1<\ldots <c_{N+1}=\infty$ such that $K(I_n)=\varepsilon$ for all $I_n=(c_n,c_{n+1}),$ $n=1,\ldots,N,$ but $K(I_0)\le\varepsilon.$ Then
\begin{equation}
a_{N+1}(\mathcal{L})\le \begin{cases}\varepsilon(N+1)^{1/r}, & 1<q<1<p<\infty\\\varepsilon(N+1)^{(1-q)/q}, & 0<q<1=p\end{cases}.
\end{equation}
 \end{lemma}
\begin{proof}
Taking the operator \eqref{P} with $\mathrm{rank}\,P\le N$ we obtain provided $q<1$ \begin{eqnarray*}
\|\mathcal{L}f-Pf\|^q_{q}
\le \sum_{n=0}^{N}\|\mathcal{L}_{I_n}|f|\|_{q}^q\le\sum_{n=0}^{N}K(I_n)^q\|f\|_{q,I_n}^q
\le \varepsilon^q\sum_{n=0}^N\|f\|_{p,I_n}^q.
\end{eqnarray*} The Holder's inequality completes the proof of the theorem as follows \begin{equation}
\|\mathcal{L}f-Pf\|^q_{q}\le\begin{cases} \varepsilon^q\bigl(\sum_{n=0}^N\|f\|_{p,I_n}^p\bigr)^{q/p}(N+1)^{q/r}, & p>1\\\varepsilon^q\bigl(\sum_{n=0}^N\|f\|_{1,I_n}\bigr)^{q}(N+1)^{1-q}, & p=1\end{cases}.\end{equation}\end{proof}

\section{Schatten type norm estimates}\label{sn}

\subsection{Case $p>1$}
Denote $r:=pq/(p-q),$ $\theta:=p'q/(p'+q),$ $\Delta_k:=[2^{k-1},2^k],$ $\Omega(l,m):=\bigcup_{l\le k\le m-1}\Delta_k$ with integers $l<m$ and
$$\sigma_k(\delta):=\biggl(\int_{\mathbb{R}^+}[\mathrm{e}^{-x2^{k\lambda}}- \mathrm{e}^{-x2^{(k+1)\lambda}}]^\delta\mathrm{d}x\biggr)^{1/q}
\biggl(\int_{\Delta_k}v^{p'}(y)\mathrm{d}y\biggr)^{1/p'},$$
\begin{eqnarray*}J_s(l,m):&=&\biggl(\int_{\Omega(l,m)} \biggl(\int_{\mathbb{R}^+}[\mathrm{e}^{-xt^{\lambda}}- \mathrm{e}^{-x2^{m\lambda}}]^\delta\mathrm{d}x\biggr)^{s/q}\biggr.\\  \times&&\!\!\!\!\!\!\!\!\!\!\!\!\!\!\!\!\biggl.\biggl[\int_{0}^{t}\chi_{\Omega(l,m)}(y)v^{p'}(y) \mathrm{d}y\biggr]^{s/p'-1}v^{p'}(t)\mathrm{d}t\biggr)^{1/s},\ \ \ \delta=\begin{cases} 1, &q\ge 1\\ q, & q<1\end{cases},\end{eqnarray*}
$$J_s(-\infty,+\infty)=:J_s=\biggl(\int_{\mathbb{R}^+} t^{-\lambda s/q}\biggl[\int_{0}^{t}v^{p'}(y) \mathrm{d}y\biggr]^{s/p'-1}v^{p'}(t)\mathrm{d}t\biggr)^{1/s},$$ $$\Lambda_s(\delta)_{(l,m)}:=\biggl(\sum_{k=l}^{m-1}\sigma_k^s(\delta)\biggr)^{1/s},\ \ \
\Lambda_s(\delta):=\Lambda_s(\delta)_{(-\infty,+\infty)}:=\biggl(\sum_{k\in\mathbb{Z}}\sigma_k^s(\delta)\biggr)^{1/s}.$$
The result of the section reads

\begin{theorem}\label{main}
Let the operator $\mathcal{L}\colon L^p\to L^q$ be compact and $s>\theta$.\\
{\rm\textbf{(i)}}  If $1<p\le q<\infty$ then
\begin{eqnarray*}\label{star3} \biggl(\sum_{n\in\mathbb{N}} a_n^s(\mathcal{L})n^{-s/p'}\biggr)^{1/s} \le
\textrm{const}(p,q,s,\lambda)\ J_s.\end{eqnarray*}
{\rm\textbf{(ii)}} If $1<q<p<\infty$ then
\begin{eqnarray*}\label{star4} \biggl(\sum_{n\in\mathbb{N}} a_n^s(\mathcal{L})n^{-s/p'}\biggr)^{1/s} \le
\textrm{const}(p,q,s,\lambda)\begin{cases} J_s, & \theta<s\le r,\\
J_r, & s>r.\end{cases}\end{eqnarray*}
{\rm\textbf{(iii)}} If $0<q<1<p<\infty$ then
\begin{eqnarray*}\label{star5} \biggl(\sum_{n\in\mathbb{N}} a_n^s(\mathcal{L})n^{s/p-s/q}\biggr)^{1/s} \le
\textrm{const}(p,q,s,\lambda)\begin{cases} J_s, & \theta<s\le r,\\
J_r, & s>r.\end{cases}\end{eqnarray*}
\end{theorem}

Proof of the results follows from \eqref{left} and Section \ref{ub}, which needs

\subsubsection{Technical lemmas}\label{tl}

The first statement is similar to \cite[Lemma 4.4]{LS2000}.

\begin{lemma}\label{lemma4} If $0<s<\infty,$ $1<p<\infty,$ $0<q<\infty$ and $l<m$ then \begin{eqnarray}\label{left}\Lambda_s(\delta)\ll J_s,\end{eqnarray}
\begin{eqnarray}\label{right}\Lambda_s(\delta)_{(l,m)}\gg J_s(l,m).\end{eqnarray}
\end{lemma}
\begin{proof} To prove \eqref{left} we put $\mathbb{A}_s^s:=\sum_{k\in\mathbb{Z}}2^{-k\lambda s/q}\biggl(\int_{0}^{2^{k}}v^{p'}(y)\mathrm{d}y\biggr)^{s/p'}$ and notice that
 \begin{eqnarray*} \Lambda_s^s(\delta)=\sum_{k\in\mathbb{Z}} \biggl(\int_{\mathbb{R}^+}[\mathrm{e}^{-x2^{k\lambda}}-\mathrm{e}^{-x2^{(k+1)\lambda}}]^\delta\mathrm{d}x\biggr)^{s/q}
\biggl(\int_{2^{k-1}}^{2^k}v^{p'}(y)\mathrm{d}y\biggr)^{s/p'}\\
\le \sum_{k\in\mathbb{Z}}\biggl(\int_{\mathbb{R}^+}\mathrm{e}^{-\delta x2^{k\lambda}}\mathrm{d}x\biggr)^{s/q}
\biggl(\int_{0}^{2^k}v^{p'}(y)\mathrm{d}y\biggr)^{s/p'}=\delta^{-s/q}\mathbb{A}_s^s.\end{eqnarray*}
We have
\begin{eqnarray*} J_s^s=\sum_{k\in\mathbb{Z}}\int_{2^{k-1}}^{2^{k}}t^{-\lambda s/q}\biggl(\int_{0}^{t}v^{p'}(y)\mathrm{d}y\biggr)^{s/p'-1}v^{p'}(t)\mathrm{d}t\\
\ge \sum_{k\in\mathbb{Z}}2^{-k\lambda s/q}\int_{2^{k-1}}^{2^{k}}\biggl(\int_{0}^{t}v^{p'}(y)\mathrm{d}y\biggr)^{s/p'-1}v^{p'}(t)\mathrm{d}t\\
=\frac{p'}{s}\sum_{k\in\mathbb{Z}} 2^{-k\lambda s/q} \biggl(\biggl[\int_{0}^{2^{k}}v^{p'}(y)\mathrm{d}y\biggr]^{s/p'} -\biggl[\int_{0}^{2^{k-1}}v^{p'}(y)\mathrm{d}y\biggr]^{s/p'}\biggr)\\
=\frac{p'}{s}\bigl[\mathbb{A}_s^s-2^{-\lambda s/q}\mathbb{A}_s^s\bigr]=
\frac{p'}{s}\bigl[1-2^{-\lambda s/q}\bigr]\mathbb{A}_s^s
\ge \frac{\delta^{s/q}p'}{s}\bigl[1-2^{-\lambda s/q}\bigr]\Lambda_s^s(\delta).\end{eqnarray*}
Hence, $$\Lambda_s(\delta)\le \bigl[\delta^{1/q}\frac{p'}{s}(1-2^{-\lambda s/q})\bigr]^{-1/s}J_s.$$

To prove \eqref{right} note first  that
 \begin{eqnarray}\label{15}\int_{\mathbb{R}^+}[\mathrm{e}^{-x2^{k\lambda}}- \mathrm{e}^{-x2^{(k+1)\lambda}}]^\delta\mathrm{d}x  \ge \int_{2^{-(k+1)\lambda+\lambda_0}}^{2^{-k\lambda}}[\mathrm{e}^{-x2^{k\lambda}}- \mathrm{e}^{-x2^{(k+1)\lambda}}]^\delta\mathrm{d}x\nonumber\\
 \ge \bigl[\mathrm{e}^{-1}-\mathrm{e}^{-2^{\lambda_0}}\bigr]^\delta 2^{-k\lambda}[1-2^{\lambda_0-\lambda}]=:C_12^{-k\lambda}\end{eqnarray} for any $0<\lambda_0<\lambda$. Write
\begin{eqnarray*}
J_s^s(l,m)\le\sum_{k=l}^{m-1}\biggl(\int_{\mathbb{R}^+}[\mathrm{e}^{-x2^{(k-1)\lambda}}- \mathrm{e}^{-x2^{m\lambda}}]^\delta\mathrm{d}x\biggr)^{s/q}\\\times
\int_{2^{k-1}}^{2^k}\biggl(\int_{0}^{t}\chi_{\Omega(l,m)}(y) v^{p'}(y)\mathrm{d}y\biggr)^{s/p'-1}v^{p'}(t)\mathrm{d}t\\
\le\frac{p'}{s}\frac{2^{s\lambda/q}}{\delta^{s/q}}\sum_{k=l}^{m-1} 2^{-sk\lambda/q}\biggl(\int_{0}^{2^k}\chi_{\Omega(l,m)}(y)v^{p'}(y)\mathrm{d}y\biggr)^{s/p'}.\end{eqnarray*}
By \cite[Proposition 2.1]{GHS1996} \begin{eqnarray*}
\sum_{k=l}^{m-1} 2^{-sk\lambda/q}\biggl(\int_{0}^{2^k}\chi_{\Omega(l,m)}(y)v^{p'}(y)\mathrm{d}y\biggr)^{s/p'}\\
=\sum_{k=l}^{m-1} 2^{-sk\lambda/q}\biggl(\sum_{l\le j\le k}\int_{2^{j-1}}^{2^j}v^{p'}(y)\mathrm{d}y\biggr)^{s/p'}\\
\ll \sum_{k=l}^{m-1} 2^{-sk\lambda/q}\biggl(\int_{2^{k-1}}^{2^k}v^{p'}(y)\mathrm{d}y\biggr)^{s/p'}
.\end{eqnarray*} Thus, we obtain \eqref{right} with help of \eqref{15}.
\end{proof}

\begin{corollary}\label{col} We have $\Lambda_s(\delta)\approx J_s.$
\end{corollary}

The next two lemmas are also analogous to some statements from \cite{LS2000}. We will apply them when $1<p\le q<\infty.$ Therefore, we actually need $\delta=1$ and $\sigma_k(1)=\bigl[1-2^{-\lambda}\bigr]^{1/q}2^{-k\lambda/q} \biggl(\int_{\Delta_k}v^{p'}(y)\mathrm{d}y\biggr)^{1/p'}.$

\begin{lemma}\label{lemma11}
Let $k_1,k_2,k_3\in\mathbb{Z},$ $k_1<k_3,$ $k_1\le k_2\le k_3$ and $z_1\in\Delta_{k_1},$ $z_0\in\Delta_{k_2},$ $z_2\in\Delta_{k_3}.$ Then $$\bigl[z_0^{-\lambda}-z_2^{-\lambda}\bigr]^{1/q}
\biggl(\int_{z_1}^{z_0}v^{p'}(y)\mathrm{d}y\biggr)^{1/p'}\ll \max_{k_1\le k\le k_2}\sigma_k(1).$$
\end{lemma}
\begin{proof}
\begin{eqnarray*}
\bigl[z_0^{-\lambda}-z_2^{-\lambda}\bigr]^{1/q}
\biggl(\int_{z_1}^{z_0}v^{p'}(y)\mathrm{d}y\biggr)^{1/p'}\le
2^{-(k_2-1)\lambda/q}\biggl(\int_{2^{k_1-1}}^{2^{k_2}}v^{p'}(y)\mathrm{d}y\biggr)^{1/p'}\\
\le \bigl[1-2^{-\lambda}\bigr]^{-1/q}2^{-(k_2-1)\lambda/q} \bigl[2^{p'k_1\lambda/q}\sigma_{k_1}^{p'}(1)+\ldots +
2^{p'k_2\lambda/q}\sigma_{k_2}^{p'}(1)\bigr]^{1/p'}\\
\le\frac{2^{\lambda/q}\bigl[1-2^{-\lambda}\bigr]^{-1/q}}{(1- 2^{-p'\lambda/q})^{1/p'}}\max_{k_1\le k\le k_2}\sigma_k(1) .\end{eqnarray*}
\end{proof}

\begin{lemma}\label{lemma12}
Let $\theta=p'q/(p'+q),$ $I_n=(c_n,c_{n+1}),$ $z_n\in I_n$ and $2^{k-1}<c_1<\ldots<c_{l+1}<2^k.$ Then $$\sum_{n=1}^l \bigl[z_n^{-\lambda}-c_{n+1}^{-\lambda}\bigr]^{\theta/q}\biggl(\int_{c_n}^{z_n}v^{p'}(y)\mathrm{d}y\biggr)^{\theta/p'}
\ll \sigma_k^\theta(1).$$
\end{lemma}
\begin{proof} By H\"{o}lder's inequality with $(p'+q)/q$ and $(p'+q)/p'$ we obtain
\begin{eqnarray*}
\sum_{n=1}^l \bigl[z_n^{-\lambda}-c_{n+1}^{-\lambda}\bigr]^{\theta/q}\biggl(\int_{c_n}^{z_n}v^{p'}(y)\mathrm{d}y\biggr)^{\theta/p'}\\
\le \sum_{n=1}^l \bigl[c_n^{-\lambda}-c_{n+1}^{-\lambda}\bigr] ^{p'/(p'+q)}\biggl(\int_{c_n}^{c_{n+1}}v^{p'}(y)\mathrm{d}y\biggr)^{q/(p'+q)}\\
\le \biggl(\sum_{n=1}^l \bigl[c_n^{-\lambda}-c_{n+1}^{-\lambda}\bigr] \biggr)^{\theta/q}\biggl(\sum_{n=1}^l \int_{c_n}^{c_{n+1}}v^{p'}(y)\mathrm{d}y\biggr)^{\theta/p'}
\ll \sigma_k^\theta(1).\end{eqnarray*}
\end{proof}

The next statement is an analog of the previous Lemma but working in the case $0<q<p<\infty.$
\begin{lemma}\label{lemma13}
Let $\theta=p'q/(p'+q),$ $I_n=(c_n,c_{n+1})$ and $2^{k-1}<c_1<\ldots<c_{l+1}<2^k.$ Then \begin{eqnarray*}\sum_{n=1}^l \biggl(\int_{c_n}^{c_{n+1}}
\biggl[\int_{c_n}^{t}v^{p'}(y)\mathrm{d}y\biggr]^{r/p'} \mathrm{d}\biggl[-
\biggl(\int_{\mathbb{R}^+}\bigl[\mathrm{e}^{-xt^\lambda}- \mathrm{e}^{-xc_{n+1}^\lambda}\bigr]^\delta\mathrm{d}x\biggr)^{r/q}\biggr]\biggr)^{\theta/r}\\
\ll [2^\lambda-1]^{\theta/q}\sigma_k^\theta(\delta).\end{eqnarray*}
\end{lemma}
\begin{proof} We have by H\"{o}lder's inequality
\begin{eqnarray*}
\sum_{n=1}^l \biggl(\int_{c_n}^{c_{n+1}}
\biggl[\int_{c_n}^{t}v^{p'}(y)\mathrm{d}y\biggr]^{r/p'} \mathrm{d}\biggl[-
\biggl(\int_{\mathbb{R}^+}\bigl[\mathrm{e}^{-xt^\lambda}- \mathrm{e}^{-xc_{n+1}^\lambda}\bigr]^\delta\mathrm{d}x\biggr)^{r/q}\biggr]\biggr)^{\theta/r}\\
\le \sum_{n=1}^l \biggl(\int_{c_n}^{c_{n+1}}v^{p'}(y)\mathrm{d}y\biggr)^{\theta/p'}
\biggl(\int_{c_n}^{c_{n+1}}\mathrm{d}\biggl[-
\biggl(\int_{\mathbb{R}^+}\bigl[\mathrm{e}^{-xt^\lambda}- \mathrm{e}^{-xc_{n+1}^\lambda}\bigr]^\delta\mathrm{d}x\biggr)^{r/q}\biggr]\biggr)^{\theta/r}\\
=\sum_{n=1}^l \biggl(\int_{c_n}^{c_{n+1}}v^{p'}(y)\mathrm{d}y\biggr)^{\theta/p'}
\biggl(\int_{\mathbb{R}^+}\bigl[\mathrm{e}^{-xc_n^\lambda}-\mathrm{e}^{-xc_{n+1}^\lambda}\bigr]^\delta\mathrm{d}x\biggr) ^{\theta/q}\ll\sigma_k^\theta(\delta).\end{eqnarray*}
\end{proof}

\subsubsection{Upper estimates for Schatten type norms}\label{ub}

\begin{theorem}\label{theorem2}
Let $1<p,q<\infty,$ $\theta=p'q/(p'+q)$ and $r=pq/(p-q).$ Assume $\mathcal{L}\colon L^p\to L^q$ is compact and $s>\theta.$\\
{\rm\textbf{(i)}} If either $1<p\le q<\infty,$ $s>\theta$ or $1<q<p<\infty,$ $\theta<s\le r$ then
\begin{eqnarray*}\biggl(\sum_{n\in\mathbb{N}} \bigl[a_n(\mathcal{L})n^{-1/p'}\bigr]^s\biggr)^{1/s}\le C(p,q,s,\lambda)\Lambda_s(1).\end{eqnarray*}
{\rm\textbf{(ii)}} If $1<q<p<\infty$ and $s>r$ then
\begin{eqnarray*}\biggl(\sum_{n\in\mathbb{N}} \bigl[a_n(\mathcal{L})n^{-1/p'}\bigr]^s\biggr)^{1/s}\le C(p,q,s,\lambda)\Lambda_r(1).\end{eqnarray*}
{\rm\textbf{(iii)}} If $0<q<1<p<\infty$ then
\begin{eqnarray*}\biggl(\sum_{n\in\mathbb{N}} \bigl[a_n(\mathcal{L})n^{1/p-1/q}\bigr]^s\biggr)^{1/s}\le C(p,q,s,\lambda)\begin{cases}\Lambda_s(q), &\theta<s\le r\\ \Lambda_r(q), &s>r\end{cases}.\end{eqnarray*}\end{theorem}

\begin{proof}
Let $0<\varepsilon< \|\mathcal{L}\|_{p\to q}$ and points $0=c_0<c_1<\ldots<c_{N+1}=\infty,$ $N=N(\varepsilon)$ be chosen so that $K(I_n)=\varepsilon$ if $I_n:=(c_n,c_{n+1}),$ $n=0,\ldots,N-1,$ and $K(I_N)\le\varepsilon.$
Then for any $c_n$ with $n=1,\ldots,N$ there exists an integer $k(n)$ such that $c_n\subset\Delta_{k(n)}=[2^{k(n)-1},2^{k(n)}].$

For all intervals $I_n$ from $n=1$ to $n=N-1$ we have two only choices:
\begin{enumerate}
\item[($\ast$)] two neighbour points, say, $c_{n_0}$ and $c_{n_0+1},$ have different $k(n_0)$ and $k(n_0+1),$ that is  $k(n_0)<k(n_0+1)$ for some $n_0$;
\item[($\ast\ast$)] two or more neighbour points, say $c_{n_1},\ldots,c_{n_1+l_{n_1}-1}$ with $l_{n_1}>1,$ are in the same interval $\Delta_k=[2^{k-1},2^k],$ that is $k(n_1)=k(n_1+1)=\ldots=k(n_1+l_{n_1}-1)$ and $I_i\subset [2^{k(n_1)-1},2^{k(n_1)}]$ with $n_1\le i\le n_1+l_{n_1}-1$ and $l_{n_1}>1.$
\end{enumerate}

Let $1<p\le q<\infty.$
Using Lemmas \ref{lemma1} and either \ref{lemma11} or \ref{lemma12} we obtain
$$\varepsilon=K(I_{n})\le{\bar \gamma}_0 A_0(1)_{I_{n}}\le C_2 \sup_{k(n)\le k\le k(n+1)} \sigma_k(1)=:C_2\sigma_{k_{n}}(1)\ \ \ \ \textrm{if}\ \ \ \ n\in(\ast)$$ and
$$\varepsilon^\theta l_{n}=\sum_{i=n}^{n+l_{n}-1}K(I_{i})^\theta\le C_3 \sigma_{k_{n}}^\theta(1)\ \ \ \ \textrm{if}\ \ \ \ n\in(\ast\ast).$$
As for the interval $I_0=[c_0,c_{1}],$ the estimate \eqref{z} yields
\begin{eqnarray*}\frac{\varepsilon}{{\bar \gamma}_0}\le\sup_{0<t<c_1}t^{-\lambda p'/q}\int_{0}^tv^{p'}(y)\mathrm{d}y \\\le \sup_{-\infty<k\le k(1)}\sup_{2^{k-1}<t<2^{k}}t^{-\lambda p'/q} \int_{0}^tv^{p'}(y)\mathrm{d}y\\
\le 2^{\lambda p'/q} \sup_{-\infty<k\le k(1)}2^{-\lambda p'k/q} \sum_{-\infty< m\le k}\int_{2^{m-1}}^{2^m}v^{p'}(y)\mathrm{d}y\end{eqnarray*}\begin{eqnarray*}
=2^{\lambda p'/q}[1-2^{-\lambda}]^{-p'/q}  \sup_{-\infty<k\le k(1)}2^{-\lambda p'k/q} \sum_{-\infty< m\le k}2^{\lambda p' m/q}\sigma_m^{p'}(1)\\
\le 2^{\lambda p'/q}[1-2^{-\lambda}]^{-p'/q} \sup_{-\infty<m\le k(1)}\sigma_m^{p'}(1)  \sup_{-\infty<k\le k(1)}\sum_{-\infty< m\le k}2^{\lambda p' (m-k)/q}\\ \le 2^{2\lambda p'/q}[1-2^{-\lambda}]^{-2p'/q} \sup_{-\infty<m\le k(1)}\sigma_m^{p'}(1) =:\sigma_{k(0)}(1). \end{eqnarray*}

Thus, \begin{eqnarray*}N(\varepsilon)=\sharp\,\{n\in\mathbb{N}\colon \sigma_{k(n)}\gg\varepsilon\} +\sum_{n\colon l_n>1}\sharp\, \{n\in\mathbb{N}\colon \sigma_{k(n)}\gg\varepsilon l_n^{1/\theta}\}\\\le\sum_{n=1}^\infty\sharp\, \{n\in\mathbb{N}\colon \sigma_{k(n)}\gg n^{1/\theta}\varepsilon\} \le \sum_{n=1}^\infty\sharp\, \{k\in\mathbb{Z}\colon \sigma_{k}\gg n^{1/\theta}\varepsilon\}.\end{eqnarray*}
On the other side we have from Lemma \ref{lemma2} \begin{equation*}\sharp\,\{n\in\mathbb{N}\colon a_n(\mathcal{L})n^{-1/p'}>\varepsilon\}\le N(\varepsilon).\end{equation*} Thus, we have by \cite[Proposition II.1.8]{BenShar} and in view of $\theta<s$
\begin{eqnarray*}\sum_{n\in\mathbb{N}} \bigl[a_n(\mathcal{L})n^{-1/p'}\bigr]^s=s\int_0^\infty t^{s-1}\sharp\,\{n\in\mathbb{N}\colon a_n(\mathcal{L})n^{-1/p'}>t\}\mathrm{d}t\\\le s\int_0^\infty t^{s-1}N(t)\mathrm{d}t \le
s\int_0^\infty \sum_{n=1}^\infty t^{s-1}\sharp\,\{k\in\mathbb{Z}\colon \sigma_k\ge \frac{n^{1/\theta}t}{C_4}\}\mathrm{d}t\\
=sC_4^s\int_0^\infty \tau^{s-1}\sum_{n=1}^\infty n^{-s/\theta}\sharp\,\{k\in\mathbb{Z}\colon \sigma_k\ge \tau\}\mathrm{d}\tau\\
=C_4^s\sum_{n=1}^\infty n^{-s/\theta}\sum_{k\in\mathbb{Z}}\sigma_k^s(1)\le  C_5^s \Lambda_s^s(1).\end{eqnarray*}

Now let $0<q<p<\infty,$ $p>1$ and consider the cases ($\ast$) and ($\ast\ast$) for $n=1,\ldots,N-1.$ If $n\in(\ast)$ then we have by \eqref{zz}, \eqref{5} and \eqref{right} from Lemma \ref{lemma4}
\begin{eqnarray*}\varepsilon=K(I_{n})\ll B_0(\delta)_{I_{n}}\le C_6 \biggl(\sum_{k(n)\le k\le k(n+1)}\sigma_k^r(\delta)\biggr)^{1/r}\\\ll C_6\Lambda_r(\delta)_{(k(n), k(n+1))}=:C_6\Sigma_{k(n)}.\end{eqnarray*}
If $n\in(\ast\ast)$ then Lemma \ref{lemma13} implies
$$\varepsilon^\theta l_{n}=\sum_{i=n}^{n+l_{n}-1}K(I_{i})^\theta\le C_7 \sigma_{k(n)}^\theta(\delta).$$
If $n=0$ then combination \eqref{zz} or \eqref{5} with \eqref{right} gives
\begin{eqnarray*}\varepsilon=K(I_0)\le C_6 B_0(\delta)_{I_0}\le C_6 \biggl(\sum_{-\infty< k\le k_{1}}\sigma_k^r(\delta)\biggr)^{1/r}\\\ll C_6\Lambda_r(\delta)_{(k(0), k(1))}=:C_6\Sigma_{k(0)},\end{eqnarray*} where $k(0)=-\infty.$
Now put $\mathbb{Z}^\ast:=\{k(n)\in\mathbb{Z}\colon n\in (\ast)\}$ and arrange sets $\mathbb{Z}^\ast_1$ and $\mathbb{Z}^\ast_2$ as follows: $$\mathbb{Z}^\ast_1:=\{k(n)\in\mathbb{Z}^\ast, n=0,2,4,\ldots\},\ \ \ \ \
\mathbb{Z}^\ast_2:=\{k(n)\in\mathbb{Z}^\ast, n=1,3,\ldots\}.$$
Therefore, \begin{eqnarray*}N(\varepsilon)=\sum_{n\colon l_n>1}\sharp\, \{n\in\mathbb{N}\colon \sigma_{k(n)}\ge \frac{\varepsilon l_n^{1/\theta}}{C_7}\}+\sharp\,\{\ n\in\mathbb{N}\colon \Lambda_r(\delta)_{(k(n), k(n+1))}\ge\frac{\varepsilon}{C_6}\}\\ \le \sum_{n=1}^\infty\sharp\, \{k\in\mathbb{Z}\colon \sigma_{k}\ge\frac{n^{1/\theta}\varepsilon}{C_7}\}+\sum_{i=1}^2
\sharp\,\{k\in\mathbb{Z}_i^\ast\colon \Sigma_{k(n)}\ge\frac{\varepsilon}{C_6}\}.\end{eqnarray*} Note that for a fixed $i=1,2$ we have $\Sigma_{l}\cap\Sigma_{m}=\emptyset$ when $l\not=m.$

Besides, Lemmas \ref{lemma2} and \ref{lemma04} yield \begin{equation}\label{num}\sharp\,\{n\in\mathbb{N}\colon a_n(\mathcal{L})u_n>\varepsilon\}\le N(\varepsilon),\ \ \ u_n=\begin{cases}n^{-1/p'}, & q\ge 1\\ n^{-1/r}, & q<1\end{cases}.\end{equation}
This implies
\begin{eqnarray*}\sum_{n\in\mathbb{N}} \bigl[a_n(\mathcal{L})u_n\bigr]^s=s\int_0^\infty t^{s-1}\sharp\,\{n\in\mathbb{N}\colon a_n(\mathcal{L})u_n>t\}\mathrm{d}t\\\le s\int_0^\infty t^{s-1}N(t)\mathrm{d}t \le
s\int_0^\infty \sum_{n=1}^\infty t^{s-1}\sharp\,\{k\in\mathbb{Z}\colon \sigma_k\ge \frac{n^{1/\theta}t}{C_7}\}\mathrm{d}t\\
+s\sum_{i=1}^2\int_0^\infty t^{s-1}\sharp\,\{k\in\mathbb{Z}_i^\ast\colon \Sigma_{k}\ge\frac{t}{C_6}\}\mathrm{d}t\\
=sC_8^s\int_0^\infty \tau^{s-1}\sum_{n=1}^\infty n^{-s/\theta}\sharp\,\{k\in\mathbb{Z}\colon \sigma_k\ge \tau\}\mathrm{d}\tau\\
+sC_9^s\sum_{i=1}^2\int_0^\infty \tau^{s-1}\sharp\,\{k\in\mathbb{Z}_i^\ast\colon \Sigma_{k}\ge \tau\}\mathrm{d}\tau\\
=sC_8^s\zeta(s/\theta)\sum_{k\in\mathbb{Z}}\sigma_k^s(\delta)+
sC_9^s\sum_{i=1}^2 \sum_{k\in\mathbb{Z}_i^\ast}\Sigma_k^s
.\end{eqnarray*} If $s\le r$ then
\begin{eqnarray*}
\sum_{k\in\mathbb{Z}_i^\ast}\biggl(\sum_{k_{n}\le k\le k_{n+1}}\sigma_k^r(\delta)\biggr)^{s/r}
\le\sum_{k\in\mathbb{Z}_i^\ast}\sum_{k_{n}\le k\le k_{n+1}}\sigma_k^s(\delta)\le 2\sum_{k\in\mathbb{Z}} \sigma_k^s(\delta).\end{eqnarray*} In the case $s>r$ we have
\begin{eqnarray*}
\sum_{k\in\mathbb{Z}_i^\ast}\biggl(\sum_{k_{n}\le k\le k_{n+1}}\sigma_k^r(\delta)\biggr)^{s/r}
\le\biggl(\sum_{k\in\mathbb{Z}_i^\ast}\sum_{k_{n}\le k\le k_{n+1}}\sigma_k^r(\delta)\biggr)^{s/r}
\le \biggl(2\sum_{k\in\mathbb{Z}} \sigma_k^r(\delta)\biggr)^{s/r}\end{eqnarray*}
and $$\sum_{k\in\mathbb{Z}} \sigma_k^s(\delta)\le \biggl(\sum_{k\in\mathbb{Z}} \sigma_k^r(\delta)\biggr)^{s/r}.$$
Thus,
\begin{eqnarray*}\biggl(\sum_{n\in\mathbb{N}} \bigl[a_n(\mathcal{L})u_n\bigr]^s\biggr)^{1/s}\le
C_{10}\biggl(\sum_{k\in\mathbb{Z}} \sigma_k^s(\delta)\biggr)^{1/s},\ \ \ \ \ \theta<s\le r,\\
\biggl(\sum_{n\in\mathbb{N}} \bigl[a_n(\mathcal{L})u_n\bigr]^s\biggr)^{1/s}\le
C_{10}\biggl(\sum_{k\in\mathbb{Z}} \sigma_k^r(\delta)\biggr)^{1/r},\ \ \ \ \ s>r.\end{eqnarray*}
\end{proof}

\subsection{Case $p=1$}
Denote as before $\Delta_k=[2^{k-1},2^k]$ and $\Omega(l,m)=\bigcup_{l\le k\le m-1}\Delta_k$ with integers $l<m.$ Put also $\bar{v}_{a}(t):=\esup_{a<x<t}v(x),$ $\bar{v}_{\Delta_k}:=\esup_{x\in\Delta_k}v(x)$  and
$$\bar{\sigma}_k(\delta):=\bar{v}_{\Delta_k}\biggl(\int_{\mathbb{R}^+}[\mathrm{e}^{-x2^{k\lambda}}- \mathrm{e}^{-x2^{(k+1)\lambda}}]^\delta\mathrm{d}x\biggr)^{1/q},\ \ \delta=\begin{cases} 1, &q\ge 1\\ q, & q<1\end{cases},$$
\begin{eqnarray*}\bar{J}_s(l,m):&=&\biggl(\int_{\Omega(l,m)} \bigl[\esup_{0<y<t}\chi_{\Omega(l,m)}(y)v(y)\bigr]^{s}
\\&&\times\mathrm{d}\biggl[-\biggl(\int_{\mathbb{R}^+}[\mathrm{e}^{-xt^{\lambda}}- \mathrm{e}^{-x2^{m\lambda}}]^\delta\mathrm{d}x\biggr)^{s/q}\biggr]\biggr)^{1/s},\end{eqnarray*}
$$\bar{J}_s(-\infty,+\infty)=:\bar{J}_s=\biggl(\int_{\mathbb{R}^+} \bar{v}_0^s(t)
t^{-\lambda s/q-1}\mathrm{d}t\biggr)^{1/s},$$ $$\bar{\Lambda}_s(\delta)_{(l,m)}:=\biggl(\sum_{k=l}^{m-1}\bar{\sigma}_k^s(\delta)\biggr)^{1/s},\ \ \
\bar{\Lambda}_s(\delta):=\biggl(\sum_{k\in\mathbb{Z}}\bar{\sigma}_k^s(\delta)\biggr)^{1/s}.$$
The result of the section reads

\begin{theorem}\label{main'}
Let operator $\mathcal{L}\colon L^1\to L^q$ be compact.\\
{\rm\textbf{(i)}}  If $1=p\le q<s<\infty$ then
\begin{eqnarray*}\label{star3'} \biggl(\sum_{n\in\mathbb{N}} a_n^s(\mathcal{L})\biggr)^{1/s} \le
\textrm{const}(q,s,\lambda)\ \bar{J}_s.\end{eqnarray*}
{\rm\textbf{(ii)}} If $0<q<1=p$ then
\begin{eqnarray*}\label{star5} \biggl(\sum_{n\in\mathbb{N}} a_n^s(\mathcal{L})n^{s-s/q}\biggr)^{1/s} \le
\textrm{const}(q,s,\lambda)\begin{cases} \bar{J}_s, & q<s\le q/(1-q),\\
\bar{J}_{q/(1-q)}, & s>q/(1-q).\end{cases}\end{eqnarray*}
\end{theorem}

Proof of the theorem follows very similar to Section \ref{ub} by using technical statements below instead of Lemmas \ref{lemma4} and \ref{lemma11} -- \ref{lemma13} respectively.

\begin{lemma}\label{lemma4'} If $0<q<s<\infty$ and $l<m$ then \begin{eqnarray}\label{left'}\bar{\Lambda}_s(\delta)\ll \bar{J}_s\end{eqnarray} and
\begin{eqnarray}\label{right'}\bar{\Lambda}_s(q)_{(l,m)}\gg \bar{J}_s(l,m), \ \ \ \ \ 0<q\le 1.\end{eqnarray}
\end{lemma}
\textit{Proof} is analogous to the proof of Lemma \ref{lemma4}. $\square$

\begin{remark} Formulas \eqref{left'} and \eqref{right'} are true for all $0<s,q<\infty.$ Moreover, similar to Collorary \ref{col} we have $$\bar{\Lambda}_s(\delta)\approx \bar{J}_s.$$
\end{remark}

The following two Lemmas are concerned the case $1=p\le q<\infty.$ This means that we need $\delta=1$ and  $\bar{\sigma}_k(1)=\bigl[1-2^{-\lambda}\bigr]^{1/q}2^{-k\lambda/q}\bar{v}_{\Delta_k}.$ Proofs of these lemmas are similar to \ref{lemma11}, \ref{lemma12}.
\begin{lemma}\label{lemma11'}
Let $k_1,k_2,k_3\in\mathbb{Z},$ $k_1<k_3,$ $k_1\le k_2\le k_3$ and $z_1\in\Delta_{k_1},$ $z_0\in\Delta_{k_2},$ $z_2\in\Delta_{k_3}.$ Then $$\bigl[z_0^{-\lambda}-z_2^{-\lambda}\bigr]^{1/q}
\esup_{z_1\le y\le z_0}v(y)\ll \max_{k_1\le k\le k_2}{\bar \sigma}_k(1).$$
\end{lemma}

\begin{lemma}\label{lemma12'}
Let $I_n=(c_n,c_{n+1}),$ $z_n\in I_n$ and $2^{k-1}<c_1<\ldots<c_{l+1}<2^k.$ Then $$\sum_{n=1}^l \bigl[z_n^{-\lambda}-c_{n+1}^{-\lambda}\bigr]\bigl[\esup_{c_n\le y\le z_n}v(y)\bigr]^{q}
\ll {\bar \sigma}_k^q(1).$$
\end{lemma}
We complete the row of technical lemmas and the section by a statement which is working when $0<q<1=p.$
\begin{lemma}\label{lemma13'}
Let $I_n=(c_n,c_{n+1})$ and $2^{k-1}<c_1<\ldots<c_{l+1}<2^k.$ Then we have for $0<q<1$ \begin{eqnarray*}&&\sum_{n=1}^l \biggl(\int_{c_n}^{c_{n+1}}\bigl[\esup_{c_n\le y\le t}v(y)\bigr]^{q/(1-q)}\biggr.\\&&\biggl.\times \mathrm{d}\biggl[-
\biggl(\int_{\mathbb{R}^+}\bigl[\mathrm{e}^{-xt^\lambda}- \mathrm{e}^{-xc_{n+1}^\lambda}\bigr]^q\mathrm{d}x\biggr)^{1/(1-q)}\biggr]\biggr)^{1-q}\ll {\bar \sigma}_k^{q}(q).\end{eqnarray*}
\end{lemma}

\section{Asymptotic estimates}\label{ab}

Put  $1\le p\le q<\infty,$ $\theta=p'q/(p'+q)$ and define a Hardy integral operator \begin{equation}\label{hardyy} H_{I;v,w}f(t):=w(t)\int_a^t f(y)v(y)\mathrm{d}y,\ \ \ \ \ t\in I,\end{equation} with weight function $w(t)=t^{-(\lambda+1)/q}.$  In the case $q\ge 1$ Lemma \ref{dec} implies \begin{equation}\label{LH} K(I)=\|\mathcal{L}_I\|_{L^p(I)\to L^q}\ll \|H_{I;v,w}\|_{L^p(I)\to L^q(I)}.\end{equation}

If we take non-negative constants $\xi$ and $\zeta$ instead of weights $v,w\ge 0$ then we obtain for the norm $\|H_{I;1,1}\|_{L^p(I)\to L^q(I)}$ by changing variables: \begin{eqnarray}\label{H0}\|H_{I;\xi,\zeta}\|_{L^p(I)\to L^q(I)}=\alpha_{pq}\,\xi\,\zeta\,\bigl|b-a\bigr|^{1-1/p+1/q},
\end{eqnarray} where $$\alpha_{pq}=\sup_{\|f\chi_{[0,1]}\|_p\le 1}\biggl(\int_0^1\biggl|\int_0^tf(s)\mathrm{d}s\biggr|^q\mathrm{d}t\biggr)^{1/q}.$$ If $1<p\le q<\infty$ then $\alpha_{p,q}=(\theta/p')^{1/p'}(\theta/q)^{1/q}.$ For $1=p\le q<\infty$ we have $\alpha_{1,q}=1.$
The next statement easy follows from triangle inequality for all $1<p\le q<\infty$:
\begin{equation}\label{tri0}\|H_{I;v_1,w_1}-H_{I;v_2,w_2}\|_{p\to q}\le \|w_1\|_{q,I}\|v_1-v_2\|_{p',I}+\|v_2\|_{p',I}\|w_1-w_2\|_{q,I}.\end{equation} If $1=p\le q<\infty$ then
\begin{equation}\label{tri1}\|H_{I;v,w_1}-H_{I;v,w_2}\|_{1\to q}\le \|v\|_{\infty,I}\|w_1-w_2\|_{q,I}.\end{equation}

The result of the section reads
\begin{theorem}\label{kr}
Let $1\le p\le q<\infty,$ $\theta:=p'q/(p'+q)$ and suppose the operator $\mathcal{L}:L^p\to L^q$ is compact.\\ {\rm\textbf{(i)}} If $p>1$ and $\sum_{k\in\mathbb{Z}}\sigma_k^\theta<\infty$ then $$\limsup_{n\to \infty}n^{1/q}a_n(\mathcal{L})\ll
\biggl(\int_{\mathbb{R}^+}t^{-(\lambda+1)\theta/q}v^\theta(t) \mathrm{d}t\biggr)^{1/\theta}.$$
{\rm\textbf{(ii)}} If $p=1,$ $v_\epsilon(t)=\lim_{\epsilon\to 0}\|v\|_{L^\infty(t-\epsilon,t+\epsilon)}$ and $\sum_{k\in\mathbb{Z}}\bar{\sigma}_k^q<\infty$ then we have
$$\limsup_{n\to \infty}n^{1/q}a_n(\mathcal{L})\ll
\biggl(\int_{\mathbb{R}^+}t^{-\lambda-1}v_\epsilon^q(t)\mathrm{d}t\biggr)^{1/q}.$$

\end{theorem}

\begin{proof}{\rm\textbf{(i)}}
Fix a large $M>0.$ Since the set of step functions is dense in a Lebesgue space then for any $\eta>0$ there exist simple functions $v_\eta,w_\eta$ on $[1/M,M]=\bigcup_{j=1}^{m_\eta}U_j$ such that $$\|v-v_\eta\|_{p',[1/M,M]}\le\eta,\ \ \ \ \ \ \ \ \ \ \|w-w_\eta\|_{q,[1/M,M]}\le\eta.$$ We may assume $v_\eta=\sum_{j=1}^{m_\eta}\xi_j\chi_{U_j}$ and $w_\eta=\sum_{j=1}^{m_\eta}\zeta_j\chi_{U_j}$ with disjoint $U_j.$

Now given $0<\varepsilon<\|\mathcal{L}\|_{p\to q}$ consider a finite sequence $0=c_0<c_1<\ldots<c_{N}<c_{N+1}=\infty,$ $N=N(\varepsilon),$ of $c_n$ chosen  as follows: $$K(c_n,c_{n+1})=K(I_n)=\varepsilon,\ \ \ \ \ n=0,1,\ldots,N-1; \ \ \ \ \ \ \ \ \ \ K(I_N)\le\varepsilon.$$ Denote $$N(M,\varepsilon)=\max\{n\in\mathbb{N}\colon 1/M<c_n\le M\},$$ $$N(1/M,\varepsilon)=\min\{n\in\mathbb{N}\colon 1/M\le c_n<M\}$$ and $$I_n(M)=\begin{cases} [c_n,M], & n=N(M,\varepsilon)\\ [1/M,c_n], & n=N(1/M,\varepsilon)\end{cases}.$$

It was shown in \cite[p.67]{EEH1997} and in \cite{LS2000} that if $U_j$ is contained in some $I_n$ for all $\varepsilon>0$ then $K(U_j)=0,$ and, therefore, $vw=0$ a.e. on $U_j.$ So, fixed $\eta>0$ let $\varepsilon_j=\inf\{\varepsilon>0\colon\,\textrm{there exists}\ n\ \textrm{such that}\,I_n\supset U_j\}$ and put $\epsilon=\min\{\varepsilon_j\colon \varepsilon_j>0\}.$ Then if $0<\varepsilon<\epsilon$ it follows that $U_j\nsubseteq I_n$ for all $n,j.$

Now we estimate the difference $$\biggl|\int_{1/M}^M v^\theta(t)w^\theta(t)\mathrm{d}t- \int_{1/M}^M \sum_{j=1}^{m_\eta}\xi_j^\theta \zeta_j^\theta\chi_{U_j}(t)\mathrm{d}t\biggr|=
\biggl|\int_{1/M}^M \bigl[v(t)w(t) -v_\eta(t)w_\eta(t)\bigr]\mathrm{d}t\biggr|$$ with the sum over those $j\in\{1,\ldots,m_\eta\}$ for which $\int_{U_j}v(t)w(t)\mathrm{d}t\not=0.$ If
$1<p=q<\infty$ then $\theta=1$ and we have from \cite[p.67]{EEH1997} that
\begin{eqnarray*}
\biggl|\int_{1/M}^M \bigl[v(t)w(t) -v_\eta(t)w_\eta(t)\bigr]\mathrm{d}t\biggr|
\le \biggl(\|v\|_{p',[1/M,M]}+\|w\|_{q,[1/M,M]}+\eta\biggr)\eta.\end{eqnarray*}
If $1<p<q<\infty$ then $\theta>1$ and we have from \cite{LS2000} that \begin{eqnarray*}&&
\biggl|\int_{1/M}^M \bigl[v^\theta(t)w^\theta(t) -v_\eta^\theta(t)w_\eta^\theta(t)\bigr]\mathrm{d}t\biggr|\\
&\le& \eta\biggl(\theta\|w\|_{q,[1/M,M]}^\theta\|v+v_\eta\|_{p',[1/M,M]}^{\theta-1}
+\theta\|v_\eta\|_{p',[1/M,M]}^\theta\|w+w_\eta\|_{q,[1/M,M]}^{\theta-1}\biggr).\end{eqnarray*}
Therefore, for all $1\le \theta<\infty$
\begin{equation}\label{eta}\biggl|\int_{1/M}^M v^\theta(t)w^\theta(t)\mathrm{d}t- \int_{1/M}^M \sum_{j=1}^{m_\eta}\xi_j^\theta \zeta_j^\theta\chi_{U_j}(t)\mathrm{d}t\biggr|\le O(\eta).\end{equation}

Since $\int_{U_j}v(t)w(t)\mathrm{d}t\not=0$ then $\varepsilon_j>0.$ Thus, we may let $0<\varepsilon<\epsilon$ and put $\mathcal{K}=\bigl\{n\colon\,\textrm{there exists}\ j\ \textrm{such that}\,I_n\subset U_j\bigr\}.$
Then $\sharp\mathcal{K}\ge N(M,\varepsilon)-N(1/M,\varepsilon)-2m_\eta$ and we obtain with help of \eqref{LH}-- \eqref{tri0} that \begin{eqnarray*}
\bigl(N(M,\varepsilon)-N(1/M,\varepsilon)-2m_\eta\bigr)\varepsilon^\theta\le \sum_{n\in\mathcal{K}}\|\mathcal{L}_{I_n}\|_{p\to q}^\theta\ll\sum_{n\in\mathcal{K}}\|H_{I_n;v,w}\|_{p\to q}^\theta\\
\ll\sum_{n\in\mathcal{K}}\|H_{I_n;v_\eta,w_\eta}\|_{p\to q}^\theta+ \sum_{j=1}^{m_\eta} \sum_{I_n\subset U_j}\|H_{U_j;\xi_j,\zeta_j}-H_{U_j;v,w}\|_{p\to q}^\theta\\
\ll \alpha_{p,q}^\theta\int_{1/M}^Mv_\eta^\theta(t) w_\eta^\theta(t)\mathrm{d}t+O(\eta^\theta).\end{eqnarray*} Applying \eqref{eta} we write
\begin{eqnarray*}
\bigl(N(M,\varepsilon)-N(1/M,\varepsilon)-2m_\eta\bigr)\varepsilon^\theta
\ll \alpha_{p,q}^\theta\int_{1/M}^Mv^\theta(t) w^\theta(t)\mathrm{d}t+O(\eta^\theta),\end{eqnarray*}
which shows
\begin{eqnarray}\label{lim0}\limsup_{\varepsilon\to 0}\varepsilon^\theta
\bigl[N(M,\varepsilon)-N(1/M,\varepsilon)\bigr]\nonumber\\
\ll \alpha_{p,q}^\theta\int_{1/M}^Mv^\theta(t) w^\theta(t)\mathrm{d}t+O(\eta^\theta).\end{eqnarray}

Now set $$k_{1/M}=\min\{k\in\mathbb{Z}\colon 2^k>N(1/M,\varepsilon)\ \textrm{and}\ K(0,2^k)\ge\varepsilon\},$$ $$k_{M}=\max\{k\in\mathbb{Z}\colon 2^k<N(M,\varepsilon)\ \textrm{and}\ K(2^k,\infty)\ge\varepsilon\}$$ and let $N=N(\varepsilon)=N(\infty,\varepsilon)$ be the total number of intervals $I_n.$ Obviously, $k_{1/M}\to-\infty,$ $k_M\to+\infty$ as $M\to +\infty.$

Since $1<p\le q<\infty$ we find from \eqref{z} and Lemmas \ref{lemma11}, \ref{lemma12} that
\begin{equation}\label{hvost1}
[N(\infty,\varepsilon)-N(M,\varepsilon)]\,\varepsilon^\theta\ll \sum_{k\ge k_M}\sigma_k^\theta,\ \ \
N(1/M,\varepsilon)\,\varepsilon^\theta\ll \sum_{k\le k_{1/M}}\sigma_k^\theta.
\end{equation}
Thus and by \eqref{lim0} \begin{eqnarray*}
\limsup_{\varepsilon\to 0}\varepsilon^\theta N(\varepsilon)\ll
\alpha_{p,q}^\theta\int_{1/M}^Mv^\theta(t) w^\theta(t)\mathrm{d}t+\sum_{|k|\ge \min\{|k_{1/M}|,|k_M|\}}\sigma_k^\theta+O(\eta^\theta),\end{eqnarray*} where the middle term is vanishing as $M\to +\infty$ provided $\sum_{k\in\mathbb{Z}}\sigma_k^\theta<\infty.$ Therefore, letting $\eta\to 0$ and $M\to+\infty$ we establish \begin{eqnarray}\label{lim1}
\limsup_{\varepsilon\to 0}\varepsilon^\theta N(\varepsilon)\ll
\alpha_{p,q}^\theta\int_{\mathbb{R}^+}v^\theta(t) w^\theta(t)\mathrm{d}t.\end{eqnarray} Further, from Lemma \ref{lemma2} we have $a_{N+1}(\mathcal{L})[N+1]^{-1/p'}\le\varepsilon.$ This gives $$\limsup_{N\to \infty}N^{1/q}a_N(\mathcal{L})\ll
\alpha_{p,q}\biggl(\int_{\mathbb{R}^+}v^\theta(t) w^\theta(t)\mathrm{d}t\biggr)^{1/\theta}.$$

{\rm\textbf{(ii)}} Put $v_\epsilon(t)=\lim_{\epsilon\to 0}\|v\|_{L^\infty(t-\epsilon,t+\epsilon)}$ and notice that by duality argument $\|H_{I,v,w}\|_{L^1(I)\to L^q(I)}=\|H_{I,v_\epsilon,w}\|_{L^1(I)\to L^q(I)}$ (see \cite[Lemma 4.1]{EHL1998} for details).

Now fix a large $M>0.$ On the strength of \cite[Lemma 4.2]{EHL1998} for each $\eta>0$ there exist step functions $w_\eta,$ $v_\eta$ on $[1/M,M]=\bigcup_{j=1}^{m_\eta}U_j$ such that $$\|w-w_\eta\|_{q,[1/M,M]}<\eta,\ \ \ \ \ \int_{[1/M,M]}w^q(t)[v_\eta^q(t)-v_\epsilon^q(t)]\mathrm{d}t<\eta$$ and $$\|v_\epsilon\|_{\infty,[1/M,M]}\ge v_\eta(t)\ge v_\epsilon(t),\ \ \ \ t\in [1/M,M].$$ We assume that
$v_\eta=\sum_{j=1}^{m_\eta}\xi_j\chi_{U_j}$ and $w_\eta=\sum_{j=1}^{m_\eta}\zeta_j\chi_{U_j}$ with disjoint $U_j.$

Now let $\varepsilon>0,$ $N=N(\infty,\varepsilon)=N(\varepsilon)$ and $\{c_n\}_0^{N+1}$ be the endpoints of the intervals $I_n=[c_n,c_{n+1}]$ with $0=c_0<c_1<\ldots<c_{N+1}=\infty$ and $$K(I_n)\le\varepsilon,\ \ \ n=0,\ldots,N;\ \ \ \ \ K(I_n,I_{n+1})>\varepsilon,\ \ \ n=0,\ldots,N-1.$$

If $q=1$ we have from \cite[p. 183]{EHL1998}
\begin{eqnarray}\label{q1}\biggl|\int_{[1/M,M]}w(t)v_\epsilon(t)\mathrm{d}t- \int_{[1/M,M]} w_\eta(t)v_\eta(t)\mathrm{d}t\biggr|\nonumber\\< \eta\bigl[1+\|v_\epsilon\|_{\infty,{[1/M,M]}}\bigr].\end{eqnarray}
If $q>1$ we obtain with help of \cite[(3.5)]{LS2000}
\begin{eqnarray}\label{qq}\biggl|\int_{[1/M,M]} w^q(t)v_\epsilon^q(t)\mathrm{d}t- \int_{[1/M,M]} w_\eta^q(t)v_\eta^q(t)\mathrm{d}t\biggr|\nonumber\\< \eta\bigl[1+q\|v_\epsilon\|_{\infty,{[1/M,M]}}\|w+w_\eta\|_{q,{[1/M,M]}}\bigr].\end{eqnarray}

Let an operator $H^\ast_{I,v,w}g(y)=v(y)\int_y^bg(t)w(t)\mathrm{d}t,$ $y\in I,$ be adjoint to $H_{I,v,w}.$
Similar to \cite[Lemma 4.5]{EHL1998} we have for $q=1$ with $\xi\ge v_\epsilon(t)\ge 0$ on $U$
\begin{eqnarray}\label{stok}\|H_{U,\xi,\zeta}\|_{1\to 1}=\|H_{U,\xi,\zeta}^\ast\|_{\infty\to \infty}\ge\|H_{U,\xi,v_\epsilon}^\ast\|_{\infty\to \infty}\nonumber\\=\|H_{U,\xi,v_\epsilon}\|_{1\to 1}\ge \frac{\xi}{2}\|(v_\epsilon\chi_{U})^\ast(t)t\|_{\infty,(0,|U|)},\end{eqnarray} where $g^\ast$ denotes the non-increasing rearrangement of a function $g$ and $|U|=\textrm{mes}\,U.$
If $q>1$ then an analog of the statement \cite[Lemma 4.5]{EHL1998} for the operator $H^\ast$ has to be modified as follows.
\begin{lemma}\label{lem}
Let $\textrm{mes}\,U<\infty$ and $\zeta,\xi\in\mathbb{R}^+$ with $\xi\ge v_\epsilon(t)\ge 0$ on $U.$ Then
\begin{eqnarray*}\|H_{U,\xi,\zeta}\|_{1\to q}= \|H^\ast_{U,\xi,\zeta}\|_{q'\to\infty}\ge\|H^\ast_{U,v_\epsilon,\zeta}\|_{q'\to\infty}\\\ge\sup_{f\in L^\infty(U),\,f\not=0}\inf_{\vartheta\in\mathbb{R}}\|H^\ast_{U,v_\epsilon,\zeta}-\vartheta v_\epsilon\|_{q'\to \infty}\\\ge\frac{\zeta}{2}\|(v_\epsilon\chi_{U})^\ast(t)t^{1/q}\|_{\infty,(0,|U|)}.
\end{eqnarray*}\end{lemma} From this and \eqref{stok} we obtain for all $q\ge 1$ that
\begin{eqnarray}\label{EHL}0\le\|H_{U,\xi,\zeta}\|_{1\to q}-\|H_{U,\xi,v_\epsilon}\|_{1\to q}\le\xi\,\zeta|U|^{1/q}-\frac{\zeta}{2}\|(v_\epsilon\chi_{U})^\ast(t)t^{1/q}\|_{\infty,(0,|U|)}\nonumber\\\le
\xi\,\zeta|U|^{1/q}-\frac{\zeta}{2}\max_{0<y<|U|^{1/q}/2}\biggl[y\biggl(\xi-\frac{2\int_{U}\zeta[\xi^q-v_\epsilon^q(t)]\mathrm{d}t} {\zeta|U|^{1/q}}\biggr)\biggr]\nonumber\\=
\xi\,\zeta|U|^{1/q}-\frac{\zeta}{2}\biggl(\xi-\frac{2\int_{U}\zeta[\xi^q-v_\epsilon^q(t)]\mathrm{d}t} {\zeta|U|^{1/q}}\biggr)\frac{|U|^{1/q}}{2}\nonumber\\
=\frac{1}{2}\int_{U}\zeta[\xi^q-v_\epsilon^q(t)]\mathrm{d}t +\frac{3}{4}\,\xi\,\zeta\,|U|^{1/q}\end{eqnarray} (see \eqref{H0} and the proof of \cite[Lemma 4.6]{EHL1998} for details).

Next, let $\mathbf{K}=\bigl\{n\colon\,\textrm{there exists}\ j\ \textrm{such that}\,I_{2n-1}\cup I_{2n}\subset U_j\bigr\}$ (see \cite[Lemma 2.3]{EHL1998} for details). Then $\sharp\mathbf{K}\ge N/2-1-2m_\eta$ and we obtain by \eqref{LH}, \eqref{H0}, \eqref{tri1} and \eqref{EHL} \begin{eqnarray*}
[N(M,\varepsilon)/2-N(1/M,\varepsilon)/2-2m_\eta]\varepsilon^q\le\sum_{n\in\mathbf{K}}K^q(I_{2n-1}\cup I_{2n})\\\ll\sum_{n\in\mathbf{K}} \|H_{I_{2n-1}\cup I_{2n},v,w}\|_{1\to q}^q
\le 3^{q-1}\sum_{n\in\mathbf{K}} \bigl[\|H_{I_{2n-1}\cup I_{2n},\xi,\zeta}\|_{1\to q}^q+\\\bigl|\|H_{I_{2n-1}\cup I_{2n},v_\epsilon,w}\|_{1\to q}-\|H_{I_{2n-1}\cup I_{2n},v_\epsilon,\zeta}\|_{1\to q}\bigr|^q\\
+\bigl|\|H_{I_{2n-1}\cup I_{2n},\xi,\zeta}\|_{1\to q}-\|H_{I_{2n-1}\cup I_{2n},v_\epsilon,\zeta}\|_{1\to q}\bigr|^q\bigl]\\
\ll\sum_{j=1}^{m_\eta}\xi_j^q\,\zeta_j^q\,|U_j|+\sum_{j=1}^{m_\eta}\biggl[\|w-\zeta\|_{q,U_j}^q \|v_\epsilon\|_{\infty,U_j}^q+\biggl(\int_{U_j}\zeta[\xi^q-v_\epsilon^q(t)]\mathrm{d}t\biggr)^q\biggr] \\
\le \int_{1/M}^M v_\eta^q(t) w_\eta^q(t)\mathrm{d}t+O(\eta^q).
\end{eqnarray*} With help of \eqref{q1} and \eqref{qq} this gives \begin{eqnarray*}
[N(M,\varepsilon)/2-N(1/M,\varepsilon)/2-2m_\eta]\varepsilon^q\ll\int_{1/M}^M v_\epsilon^q(t) w^q(t)\mathrm{d}t+O(\eta^q).\end{eqnarray*} Thus, \begin{eqnarray*}\limsup_{\varepsilon\to 0}\varepsilon^q\bigl[
N(M,\varepsilon)-N(1/M,\varepsilon)\bigr]\ll\int_{1/M}^M v_\epsilon^q(t) w^q(t)\mathrm{d}t+O(\eta^q).\end{eqnarray*}
Further, similar to the proof of the part {\rm\textbf{(i)}} we obtain \begin{eqnarray*}
\limsup_{\varepsilon\to 0}\varepsilon^q N(\infty,\varepsilon)\ll
\int_{\mathbb{R}^+}v_\epsilon^q(t) w^q(t)\mathrm{d}t.\end{eqnarray*} Finally, we have by the inequality $a_{N+1}(\mathcal{L})\le\varepsilon$ from Lemma \ref{lemma2} that $$\limsup_{N\to \infty}N^{1/q}a_N(\mathcal{L})\ll
\biggl(\int_{\mathbb{R}^+}v_\epsilon^q(t) w^q(t)\mathrm{d}t\biggr)^{1/q}.$$
\end{proof}

\end{document}